\documentclass[12pt,a4paper]{article}
\usepackage{amsfonts}
\usepackage{amsmath}
\usepackage{amssymb}

\newtheorem{theorem}{Theorem}

\newtheorem{condition}[theorem]{Condition}

\newtheorem{corollary}[theorem]{Corollary}

\newtheorem{definition}[theorem]{Definition}

\newtheorem{lemma}[theorem]{Lemma}

\newtheorem{proposition}[theorem]{Proposition}
\newtheorem{remark}[theorem]{Remark}

\newenvironment{proof}[1][Proof]{\noindent\textbf{#1.} }{\ \rule{0.5em}{0.5em}}

\input{tcilatex}
\begin{document}

\title{Stochastic differential equations in a scale of Hilbert spaces 2.
Global solutions.}
\author{Georgy Chargaziya \and Alexei Daletskii \\
Department of Mathematics, University of York, UK}
\maketitle

\begin{abstract}
A stochastic differential equation with coefficients defined in a scale of
Hilbert spaces is considered. The existence, uniqueness and path-continuity
of infinite-time solutions is proved by an extension of the Ovsyannikov
method. This result is applied to a system of equations describing
non-equilibrium stochastic dynamics of (real-valued) spins of an infinite
particle system on a typical realization of a Poisson or Gibbs point process
in ${\mathbb{R}}^{n}$. The paper improves the results of the work by the
second named author "Stochastic differential equations in a scale of Hilbert
spaces", Electron. J. Probab. 23, where finite-time solutions were
constructed.
\end{abstract}

\section{Introduction}

The purpose of this work is to study an infinite-dimensional stochastic
differential equation (SDE) 
\begin{equation}
d\xi (t)=f(\xi (t)dt+B(\xi (t))dW(t),  \label{sde}
\end{equation}%
with the coefficients $f$ and $B$ defined in a scale of densely embedded
Hilbert spaces $\left( X_{\alpha }\right) _{\alpha \in \mathcal{A}}$, where $%
\mathcal{A}$ is a real interval, and $W$ is a cylinder Wiener process on a
fixed Hilbert space $\mathcal{H}$. That is, $f$ and $B$ are Lipschitz
continuous maps $X_{\alpha }\rightarrow X_{\beta }$ and $X_{\alpha
}\rightarrow H_{\beta }:=HS(\mathcal{H},X_{\beta })$, $\beta >\alpha $,
respectively, but are not in general well-defined in any fixed $X_{\alpha }$%
, with the corresponding Lipschitz constants $L_{\alpha \beta }$ becoming
infinite as $\left\vert \alpha -\beta \right\vert \rightarrow 0$. Here $HS(%
\mathcal{H},X_{\beta })$ stands for the space of Hilbert-Schmidt operators $%
\mathcal{H}\rightarrow X_{\beta }$.

Equation (\ref{sde}) cannot be treated by methods of the classical theory of
SDEs in Banach spaces (see e.g. \cite{DaFo} and \cite{DZ}), because its
coefficients are singular in any fixed $X_{\alpha }$. Some progress can be
achieved if 
\begin{equation}
L_{\alpha \beta }\sim (\beta -\alpha )^{-1/q}\ \text{as }\left\vert \alpha
-\beta \right\vert \rightarrow 0,  \label{OES}
\end{equation}%
with $q=2$. Under this condition, a solution with initial value in $%
X_{\alpha }$ exists in $X_{\beta }$ , with lifetime $T_{\alpha \beta }\sim
(\beta -\alpha )^{1/2}$, see \cite{Dal}. This result generalizes the
Ovsyannikov method for ordinary differential equations, see e.g. \cite{Deim}%
, \cite{BHP} and \cite{DaF}, in which setting it is sufficient to assume
that $q=1$.

It has been noticed in \cite{DaF} that, in case of $q>1,$ a solution of the
ODE 
\begin{equation*}
\frac{d}{dt}u(t)=f(u(t)),\ u(0)\in X_{\alpha },
\end{equation*}%
exists in any $X_{\beta }$, $\beta >\alpha $, with infinite lifetime. In the
present paper, we build upon the ideas of \cite{DaF}, which enable us to
generalize the results of \cite{Dal} and prove the existence and uniqueness
of a global solution $\xi (t)$ of equation (\ref{sde}) in any $X_{\beta }$, $%
\beta >\alpha $, with initial value $\xi (0)\in X_{\alpha }$, provided (\ref%
{OES}) holds with $q>2$. Moreover, we show that $\xi (t)$ is $p$-integrable
for any $p<q$ and has a continuous modification.

The structure of the paper is as follows. In Section \ref{sec-main} we
introduce the framework and notations and formulate our main existence and
uniqueness result. In Section \ref{estimates} we obtain technical estimates,
which play crucial role in what follows. Section \ref{proof} is devoted to
the proof of our main existence and uniqueness result. In Sections \ref%
{estofsol} and \ref{contin} we derive an estimate of the growth of solutions
and prove the existence of its continuous modification, respectively.

Section \ref{spins} is devoted to our main example, which is motivated by
the study of countable systems of particles randomly distributed in a
Euclidean space ${\mathbb{R}}^{n}(=:\mathfrak{X})$. Each particle is
characterized by its position $x$ and an internal parameter (spin) $\sigma
_{x}\in S={\mathbb{R}^{1}}$. For a given fixed (\textquotedblleft
quenched\textquotedblright ) configuration $\gamma $ of particle positions,
which is a locally finite subset of ${\mathbb{R}}^{n}$, we consider a system
of stochastic differential equations describing (non-equilibrium) dynamics
of spins $\sigma _{x},$ $x\in \gamma $. Two spins $\sigma _{x}$ and $\sigma
_{y}$ are allowed to interact via a pair potential if the distance between $x
$ and $y$ is no more than a fixed interaction radius $r$, that is, they are
neighbors in the geometric graph defined by $\gamma $ and $r.$ Vertex
degrees of this graph are typically unbounded, which implies that the
coefficients of the corresponding equations cannot be controlled in a single
Hilbert or Banach space (in contrast to spin systems on a regular lattice,
which have been well-studied, see e.g. \cite{DZ1} and more recent
developments in \cite{ADK}, \cite{ADK1}, \cite{INZ}, and references
therein). However, under mild conditions on the density of $\gamma $
(holding for e.g. Poisson and Gibbs point processes in ${\mathbb{R}}^{n}$),
it is possible to apply the approach discussed above and construct a
solution in the scale of Hilbert spaces $S_{\alpha }^{\gamma }$ of weighted
sequences $(q_{x})_{x\in \gamma }\in S^{\gamma }$ such that $\sum_{x\in
\gamma }\left\vert q_{x}\right\vert ^{2}e^{-\alpha \left\vert x\right\vert
}<\infty ,\ \alpha >0$. Local solutions of the above system were constructed
in \cite{Dal} by a somewhat different method.

Construction of non-equilibrium stochastic dynamics of infinite particle
systems of the aforementioned type has been a long-standing problem, even in
the case of linear drift and a single-particle diffusion coefficient. It has
become important in the framework of analysis on 
spaces $\Gamma (\mathfrak{X},S)$ of configurations $\{(x,\sigma_x)\}_{x\in%
\gamma}$ with marks (see e.g. \cite{DaVe}), and is motivated by a variety of
applications, in particular in modeling of non-crystalline (amorphous)
substances, e.g. ferrofluids and amorphous magnets, see e.g. \cite{Romano}, 
\cite[Section 11]{OHandley}, \cite{Bov} and \cite{DKKP,DKKP1}. $\Gamma (%
\mathfrak{X},S)$ possesses a fibration-like structure over the space $\Gamma
(\mathfrak{X})$ of position configurations $\gamma$, with the fibres
identified with $S^{\gamma}$, see \cite{DKKP}. Thus the construction of spin
dynamics of a quenched system (in $S^{\gamma}$) is complementary to that of
the dynamics in $\Gamma (\mathfrak{X})$.

Various aspects of the study of deterministic (Hamiltonian) and stochastic
evolution of configurations $\gamma \in \Gamma (\mathfrak{X})$, in its
deterministic (Hamiltonian) and stochastic form have been discussed by many
authors, see e.g. \cite{Lan,LLL,FrLO,BGSR,AKR,FKK} and references given
there. It is anticipated that (some of) these results can be combined with
the approach proposed in the present paper allowing to build stochastic
dynamics on the marked configuration space $\Gamma (\mathfrak{X},S)$. In
particular, the results of Section \ref{spins} are used in a forthcoming
paper \cite{DaFin} for the construction of a mixed-type jump diffusion
dynamics in $\Gamma (\mathfrak{X},S)$.

Finally, in Section \ref{examples} we give two further examples of the maps
satisfying condition (\ref{OES}).

Observe that the family $X_{\alpha }=S_{\alpha }^{\gamma }$, $\alpha >0$,
forms the dual to nuclear space $\Phi ^{\prime }=\cup _{\alpha }X_{\alpha }$%
. SDEs on such spaces were considered in \cite{KMW}, \cite{KJX}. The
existence of solutions to the corresponding martingale problem was proved
under assumption of continuity of coefficients on $\Phi ^{\prime }$ and
their linear growth (which, for the diffusion coefficient, is supposed to
hold in each $\alpha $-norm). Moreover, the existence of strong solutions
requires a dissipativity-type estimate in each $\alpha $-norm, too, which
does not hold in our framework.\bigskip

\textbf{Acknowledgment}. We are very grateful to Dmitri Finkelshtein and
Zdzislaw Brze\'{z}niak for their interest in this work and many stimulating
discussions.

\section{Setting and main results\label{sec-main}}

In this section we introduce the general framework we will be using. We
start with the following general definition.

Let us consider a family $\mathfrak{B}$ of Banach spaces $B_{\alpha }$
indexed by $\alpha \in \mathcal{A}:=\left[ \alpha _{\ast },\alpha ^{\ast }%
\right] $ with fixed $0\leq \alpha _{\ast },\alpha ^{\ast }<\infty $, and
denote by $\left\Vert \cdot \right\Vert _{B_{\alpha }}$ the corresponding
norms. When speaking of these spaces and related objects, we will always
assume that the range of indices is $\left[ \alpha _{\ast },\alpha ^{\ast }%
\right] $, unless stated otherwise. The interval $\mathcal{A}$ remains fixed
for the rest of this work.

\begin{definition}
The family $\mathfrak{B}$ is called a scale if 
\begin{equation}
B_{\alpha }\subset B_{\beta }\ {\text{and }}\left\Vert u\right\Vert
_{B_{\beta }}\leq \left\Vert u\right\Vert _{B_{\alpha }}{\text{ for any }}%
\alpha <\beta ,\ u\in X_{\alpha },  \label{scale}
\end{equation}%
where the embedding means that $B_{\alpha }$ is a vector subspace of $%
B_{\beta }$.
\end{definition}

We will use the following notations:%
\begin{equation*}
\overline{B}:=\dbigcup\limits_{\alpha \in \left[ \alpha _{\ast },\alpha
^{\ast }\right) }B_{a},\ \underline{B}:=\dbigcap\limits_{\alpha \in \left(
\alpha _{\ast },\alpha ^{\ast }\right] }B_{a}.
\end{equation*}

\begin{definition}
\label{Defovsop}For two scales $\mathfrak{B}_{1}$, $\mathfrak{B}_{2}$ (with
the same index set) and a constant $q>0$ we introduce the class ${\mathcal{GL%
}}_{q}(\mathfrak{B}_{1},\mathfrak{B}_{2})$ of (generalized Lipschitz) maps $%
f:\overline{B}\rightarrow \overline{B}$ such that

\begin{enumerate}
\item $f(B_{\alpha })\subset B_{\beta }$ for any $\alpha <\beta $;

\item there exists constant $L>0$ such that 
\begin{equation}
\left\Vert f(u)-f(v)\right\Vert _{\beta }\leq \frac{L}{\left\vert \beta
-\alpha \right\vert ^{1/q}}\left\Vert u-v\right\Vert _{B_{\alpha }}
\label{lg}
\end{equation}%
for any $\alpha <\beta $ and $u,v\in B_{\alpha }$.
\end{enumerate}
\end{definition}

We will write ${\mathcal{GL}}_{q}(\mathfrak{B}):={\mathcal{GL}}_{q}(%
\mathfrak{B}_{1},\mathfrak{B}_{2})$ if $\mathfrak{B}_{1}=\mathfrak{B}_{2}=:%
\mathfrak{B}$.

\begin{remark}
The constant $L$ may depend on $\alpha ^{\ast }$ and $\alpha _{\ast }$, as
usually happens in applications.
\end{remark}

\begin{remark}
\label{rem-bound}Setting $v=0$ in (\ref{lg}), we obtain the linear growth
condition 
\begin{equation}
\left\Vert f(u)\right\Vert _{B_{\beta }}\leq \frac{K}{\left\vert \beta
-\alpha \right\vert ^{1/q}}\left( 1+\left\Vert u\right\Vert _{B_{\alpha
}}\right) ,\ u\in B_{a},  \label{gc}
\end{equation}%
for some constant $K$ and any $\alpha <\beta $.
\end{remark}

\begin{remark}
Assume that $\phi $ is Lipschitz continuous in each $B_{\alpha }$ with a
uniform Lipschitz constant $M$. Then $\phi \in {\mathcal{GL}}_{q}(\mathfrak{B%
})$ with $L=\left( \alpha ^{\ast }-\alpha _{\ast }\right) ^{1/q}M$.
\end{remark}

\begin{remark}
\label{rem-scale}Some authors have used the scale $B_{\alpha }$ such that $%
B_{\alpha }\subset B_{\beta }$ if $\alpha >\beta $. That framework can be
transformed to our setting by an appropriate change of the parametrization,
e.g. $\alpha \mapsto \alpha ^{\ast }-\alpha $.
\end{remark}

In what follows, we will use the following three main scales of spaces:

\begin{enumerate}
\item[(1)] the scale $\mathfrak{X}$ of separable Hilbert spaces $X_{\alpha
}; $

\item[(2)] the scale $\mathfrak{H}$ of spaces 
\begin{equation}
H_{\alpha }\equiv HS(\mathcal{H},X_{\alpha }):=\left\{ {\text{%
Hilbert-Schmidt operators }}{\mathcal{H}}\rightarrow X_{\alpha }\right\} ,
\label{HS}
\end{equation}%
for a fixed separable Hilbert space $\mathcal{H}$;

\item[(3)] the scale $\mathfrak{Z}_{T}^{p}$ of Banach spaces $Z_{\alpha
,T}^{p}$ of progressively measurable random processes $u:\left[ 0,T\right)
\rightarrow X_{\alpha }$ with finite norm 
\begin{equation*}
\left\Vert u\right\Vert _{Z_{\alpha ,T}^{p}}:=\sup_{t\in \left[ 0,T\right)
}\left( {\mathbb{E}}\left\Vert u(t)\right\Vert _{X_{\alpha }}^{p}\right)
^{1/p}\text{{,}}
\end{equation*}%
defined on a suitable filtered probability space $\left( \Omega ,\mathcal{F}%
,P\right) $.
\end{enumerate}

Our aim is to construct a strong solution of equation (\ref{sde}), that is,
a solution of the stochastic integral equation 
\begin{equation}
u(t)=u_{0}+\int_{0}^{t}f(u(s))ds+\int_{0}^{t}B(u(s))dW(s),\ t\leq T,\
u_{0}\in X_{\alpha _{\ast }},  \label{SDE-int}
\end{equation}%
where $W(t),\ t\leq T,$ is a fixed cylinder Wiener process in $\mathcal{H}$
(cf. (\ref{HS})) defined on the probability space $\left( \Omega ,\mathcal{F}%
,P\right) $, with coefficients acting in the scale $\mathfrak{X}$ for a
fixed $p\geq 2$.

The following theorem states the main result of this paper.

\begin{theorem}[Existence and uniqueness]
\label{KeyResult}\label{theor-main} Assume that $f\in {\mathcal{GL}}_{q}(%
\mathfrak{X})$ and $B\in {\mathcal{GL}}_{q}(\mathfrak{X},\mathfrak{H})$, $%
q>2 $ and $u_{0}\in X_{\alpha }$, $\alpha \in \mathcal{A}$. Then, for any $%
T>0$, the following holds:

\begin{enumerate}
\item[(1)] equation (\ref{SDE-int}) has a unique solution $u\in Z_{\alpha
^{\ast },T}^{2}$;

\item[(2)] $u\in Z_{\beta ,T}^{p}$ for any $p\in \left[ 2,q\right) $ and $%
\beta >\alpha $;

\item[(3)] $u$ has continuous sample paths a.s.
\end{enumerate}
\end{theorem}

The proof of the first two statements is given in Section \ref{EU} below. We
will show that the map $u\mapsto \mathcal{T}(u)$, where%
\begin{equation}
\mathcal{T}(u)(t)=u_{0}+\int_{0}^{t}f(u(s))ds+\int_{0}^{t}B(u(s))dW(s),\
t\in \lbrack 0,T],  \label{lip-map}
\end{equation}%
has a unique fixed point in $Z_{\beta ,T}^{p}$ for any $\beta >\alpha $, by
Picard iterative process. The third statement is proved in Section \ref%
{contin} by Kolmogorov's continuity theorem.

From now on, we keep $u_{0}$ fixed assume without loss of generality that $%
u_{0}\in X_{\alpha _{\ast }}$ (otherwise, we can always re-define the
parameter set $\mathcal{A}$). We also fix an arbitrary $T$ and write $%
Z_{\beta }^{p}$ instead of $Z_{\beta ,T}^{p}$.

\section{Main estimates\label{estimates}}

\label{AR}

In this section, we derive certain estimates of the map $\mathcal{T}$
defined by formula (\ref{lip-map}). We first observe that if $\xi \in
Z_{\alpha }^{p}$ and $\alpha <\beta $ then $\Phi (\xi )$ is a predictable $%
H_{\beta }$-valued process because $\Phi $ is continuous by inequality (\ref%
{lg}). Now using inequality (\ref{gc}) we also see that%
\begin{multline*}
\int_{0}^{T}\left\Vert \Phi (\xi (s))\right\Vert _{H_{\beta }}^{2}ds\leq
C_{1}\int_{0}^{T}\left( 1+\left\Vert \xi (s)\right\Vert _{B_{\alpha
}}^{2}\right) ds \\
\leq C_{2}\int_{0}^{T}\left( 1+\left\Vert \xi (s)\right\Vert _{B_{\alpha
}}^{p}\right) ds<\infty 
\end{multline*}%
because $\xi \in Z_{\alpha }^{p}\subset Z_{\alpha }^{2}$. Thus the
stochastic integral 
\begin{equation*}
\int_{0}^{t}\Phi (\xi (s))dW(s),\ t\leq T,
\end{equation*}%
is unambiguously defined as a square integrable $X_{\beta }$-valued
martingale. Therefore $\mathcal{T}$ is a well-defined map $\overline{Z\,}%
\rightarrow Z_{\alpha ^{\ast }}$.

\begin{theorem}
Assume that $f\in {\mathcal{GL}}_{q}(\mathfrak{X})$ and $B\in {\mathcal{GL}}%
_{q}(\mathfrak{X},\mathfrak{H})$. Then $\mathcal{T\in }{\mathcal{GL}}_{q}(%
\mathfrak{Z}^{p})$ for any $p\geq 2$.
\end{theorem}

\begin{proof}
Let us fix $p\geq 2$ and $\alpha ,\beta \in \mathcal{A}$, $\alpha <\beta $.
For simplicity, we will use the shorthand notation $Z_{\alpha }:=Z_{\alpha
,T}^{p}$.

We first show that $\mathcal{T}(Z_{\alpha })\subset Z_{\beta }$. Observe
that $f(u(s))\in X_{\beta }$ and $B(u(s))\in H_{\beta }\ $for any $s\in
\lbrack 0,T],$ and the integrals in the right-hand side of (\ref{lip-map})
are well-defined in $X_{\beta }$. The inclusion in question immediately
follows from the properties of those integrals.

Now we shall show that condition (\ref{lg}) of the Definition \ref{Defovsop}
also holds. Introduce notations 
\begin{equation*}
\bar{F}(s):=F(\xi _{1}(s))-F(\xi _{2}(s))\text{ and }\bar{\Phi}(s):=\Phi
(\xi _{1}(s))-\Phi (\xi _{2}(s)),\ s\in \lbrack 0,T],
\end{equation*}%
we obtain%
\begin{multline}
\mathbb{E}\left[ ||\mathcal{T}(\xi _{1})(t)-\mathcal{T}(\xi
_{2})(t)||_{X_{\beta }}^{p}\right] =\mathbb{E}\left[ ||\int_{0}^{t}\bar{F}%
(s)ds+\int_{0}^{t}\bar{\Phi}(s)dW(s)||_{X_{\beta }}\right] ^{p} \\[0.01in]
\leq 2^{p-1}\mathbb{E}\left[ \int_{0}^{t}||\bar{F}(s)||_{X_{\beta }}ds\right]
^{p}+2^{p-1}\mathbb{E}\left[ ||\int_{0}^{t}\bar{\Phi}(s)dW(s)\ ||_{X_{\beta
}}\right] ^{p}.  \label{EqAux1}
\end{multline}%
Applying the H\"{o}lder inequality, well-known formula for the moments of
Ito integral (see e.g. \cite{DaF}) and estimate (\ref{lg}) to inequality (%
\ref{EqAux1}) above we obtain the estimate{\ \fontsize{10}{7.2}\selectfont%
\begin{multline}
\mathbb{E}\left[ ||\mathcal{T}(\xi _{1})(t)-\mathcal{T}(\xi
_{2})(t)||_{X_{\beta }}^{p}\right] \leq 2^{p-1}t^{p-1}\int_{0}^{t}\mathbb{E~}%
||\bar{F}(s)||_{X_{\beta }}^{p}ds \\
+2^{p-1}\left[ \frac{p}{2}(p-1)\right] ^{p/2}t^{p/2-1}\int_{0}^{t}\mathbb{E~}%
||\bar{\Phi}(s)||_{H_{\beta }}^{p}ds \\
\leq \frac{\hat{L}(T)}{(\beta -\alpha )^{p/q}}\int_{0}^{t}\mathbb{E~}||\xi
_{1}(s)-\xi _{2}(s)||_{X_{\alpha }}^{p}ds,  \label{2ndIntMapLCondition}
\end{multline}%
where} $\hat{L}(T)=(T^{p-1}+\left[ \frac{p}{2}(p-1)\right]
^{p}T^{p/2-1})2^{p-1}L^{p}$ . Consequently for all $\alpha <\beta \in 
\mathcal{A}$ and $\xi _{1},\xi _{2}\in Z_{\alpha }$ the following general
result holds: 
\begin{multline}
||\mathcal{T}(\xi _{1})-\mathcal{T}(\xi _{2})||_{Z_{\beta }}\leq \frac{\hat{L%
}(T)T}{(\beta -\alpha )^{p/q}}\sup_{s\in \lbrack 0,T]}\mathbb{E}~||\xi
_{1}(s)-\xi _{2}(s)||_{X_{\alpha }}^{p}  \label{LipZ} \\
=\frac{\sqrt[p]{\hat{L}(T)T}}{(\beta -\alpha )^{1/q}}||\xi _{1}-\xi
_{2}||_{Z_{\alpha }},
\end{multline}%
and the proof is complete.
\end{proof}

\begin{corollary}
\label{CompositionLemma} For any $\alpha >\alpha _{\ast }$ and all $n\in 
\mathbb{N}$%
\begin{equation*}
\mathcal{T}^{n}:Z_{\alpha _{\ast }}\rightarrow Z_{\alpha },
\end{equation*}%
where $\mathcal{T}^{n}$ stands for the $n$-th composition power of $\mathcal{%
T}$.
\end{corollary}

Corollary \ref{CompositionLemma} shows in particular that given $\xi \in
Z_{\alpha _{\ast }}$ the sequence of processes $\{\mathcal{T}^{n}(\xi
)\}_{n=1}^{\infty }$ belongs to $Z_{\alpha }$ for all $\alpha >\alpha _{\ast
}$.

\begin{remark}
Observe that we have 
\begin{equation}
\sqrt[p]{\hat{L}(T)T}\leq a(T)=\left\{ 
\begin{array}{c}
a_{p}LT,\ T\geq 1 \\ 
a_{p}LT^{1/2},\ T<1%
\end{array}%
\right. \text{, }  \label{Lbound}
\end{equation}%
where $a_{p}=2^{p-1}\left( \left( \frac{p}{2}\right) ^{p/2}(p-1)+1\right) .$
\end{remark}

\begin{lemma}
For any $n\in \mathbb{N}$, $\alpha <\beta $ and $\xi _{1},\xi _{2}\in
Z_{\alpha }$ we have the estimate%
\begin{equation}
||\mathcal{T}^{n}(\xi _{1})-\mathcal{T}^{n}(\xi _{2})||_{Z_{\beta }}^{p}\leq 
\frac{n^{np/q}}{n!}\left( \frac{\hat{L}(T)T}{(\beta -\alpha )^{p/q}}\right)
^{n}||\xi _{1}-\xi _{2}||_{Z_{\alpha }}^{p}.  \label{n-est}
\end{equation}
\end{lemma}

\begin{proof}
We fix a partition of the interval $\left[ \alpha ,\beta \right] $ in $n$
intervals $\left[ \psi _{k},\psi _{k+1}\right] ,~k=0,...,n-1$, $\psi
_{0}=\alpha $, $\psi _{n}=\beta $, of equal length $\frac{\beta -\alpha }{n}$%
. Then, iterating estimate (\ref{2ndIntMapLCondition}), we obtain%
\begin{multline}
\mathbb{E}\left[ ||\mathcal{T(T}^{n-1}(\xi _{1}))(t)-\mathcal{T(T}^{n-1}(\xi
_{2}))(t)||_{X_{\beta }}^{p}\right]  \\
\leq \frac{\hat{L}(T)n^{p/q}}{(\beta -\alpha )^{p/q}}\int_{0}^{t}\mathbb{E~}%
||\mathcal{T}^{n-1}(\xi _{1}(s))-\mathcal{T}^{n-1}(\xi _{2}(s))||_{X_{\alpha
}}^{p}ds \\
\leq ...\leq \left[ \frac{\hat{L}(T)n^{p/q}}{(\beta -\alpha )^{p/q}}\right]
^{n}\int_{0}^{t}\int_{0}^{t_{1}}...\int_{0}^{t_{n-1}}\mathbb{E~}||\xi
_{1}(s)-\xi _{2}(s)||_{X_{\alpha }}^{p}dsdt_{n-1}...dt_{1},
\end{multline}%
and the result follows.
\end{proof}

\begin{corollary}
For any $n\in \mathbb{N}$, $\alpha <\beta $ and $\xi \in Z_{\alpha }$ we
have the estimates%
\begin{multline}
\mathbb{E}\left[ ||\mathcal{T}^{n}(\xi )(t)-\mathcal{T}^{n+1}(\xi
)(t)||_{X_{\beta }}^{p}\right]  \label{n-est0} \\
\leq \left[ \frac{\hat{L}(T)n^{p/q}}{(\beta -\alpha )^{p/q}}\right]
^{n}\int_{0}^{t}\int_{0}^{t_{1}}...\int_{0}^{t_{n-1}}\mathbb{E~}||\xi (s)-%
\mathcal{T(}\xi (s))||_{X_{\alpha }}^{p}dsdt_{n-1}...dt_{1}
\end{multline}%
and%
\begin{equation}
||\mathcal{T}^{n}(\xi )-\mathcal{T}^{n+1}(\xi )||_{Z_{\beta }}^{p}\leq \frac{%
n^{np/q}}{n!}\left( \frac{\hat{L}(T)T}{(\beta -\alpha )^{p/q}}\right)
^{n}||\xi -\mathcal{T}(\xi )||_{Z_{\alpha }}^{p}.  \label{n-est1}
\end{equation}
\end{corollary}

\begin{lemma}
\label{SecondCauchyLemma} Suppose $\alpha <\beta \in \mathcal{A}$ and $\xi
\in Z_{\alpha _{\ast }}$. For all $m,n\in \mathbb{N}$, $m>n$, the following
inequality holds%
\begin{equation}
||\mathcal{T}^{n}(\xi )-\mathcal{T}^{m}(\xi )||_{Z_{\beta }}\leq ||\xi -%
\mathcal{T}(\xi )||_{Z_{\alpha }}\sum_{k=n}^{m}\frac{\sqrt[p]{\hat{L}%
(T)^{k}T^{k}}}{(\beta -\alpha )^{k/q}}\frac{k^{\theta k}}{\sqrt[p]{k!}}.
\label{cauchy-est}
\end{equation}
\end{lemma}

\begin{proof}
We have%
\begin{multline*}
||\mathcal{T}^{n}(\xi )-\mathcal{T}^{m}(\xi )||_{Z_{\beta }}\leq
\sum_{k=n}^{m-1}||\mathcal{T}^{k}(\xi )-\mathcal{T}^{k+1}(\xi )||_{Z_{\beta
}} \\
\leq \sum_{k=n}^{m-1}\frac{k^{n/q}}{\left( k!\right) ^{1/p}}\left( \frac{%
\hat{L}(T)T}{(\beta -\alpha )^{p/q}}\right) ^{n/p}||\xi -\mathcal{T}(\xi
)||_{Z_{\alpha }}.
\end{multline*}%
The result is proved.
\end{proof}

Finally, we prove regularity of the right-hand side of (\ref{cauchy-est}).
In what follows, we will use the notation%
\begin{equation}
E^{(p)}(t,\varepsilon ,\theta ):=1+\sum_{n=1}^{\infty }\frac{t^{n}}{%
\varepsilon ^{\theta n}}\frac{n^{\theta n}}{\left( n!\right) ^{1/p}}
\label{E-func}
\end{equation}%
Observe that for $p=1$ and $\theta =0$ the right-hand side of (\ref{E-func})
reduces to an exponential series, so that $E^{(1)}(c,\varepsilon ,0)=e^{c}$.

\begin{lemma}
\label{SeriesLemma} For any $t,p,\varepsilon >0$ and $\theta \in \lbrack 0,%
\frac{1}{p})$ we have 
\begin{equation*}
E^{(p)}(t,\varepsilon ,\theta )<\infty .
\end{equation*}
\end{lemma}

\begin{proof}
By analyzing the ratio of terms of series (\ref{E-func}) we get 
\begin{multline*}
\lim_{n\rightarrow \infty }\frac{\frac{t^{\left( n+1\right) }}{\varepsilon
^{\theta (n+1)}}\frac{\left( n+1\right) ^{\theta n}}{\left( (n+1)!\right)
^{1/p}}}{\frac{t^{n}}{\varepsilon ^{\theta n}}\frac{n^{\theta n}}{\left(
n!\right) ^{1/p}}}=\lim_{n\rightarrow \infty }\frac{t}{\varepsilon ^{\theta }%
}(n+1)^{\theta n+\theta -\frac{1}{p}}\frac{1}{n^{\theta n}}, \\
=\lim_{n\rightarrow \infty }\frac{t}{\varepsilon ^{\theta }}\left( 1+\frac{1%
}{n}\right) _{.}^{\theta n}(n+1)^{\theta -\frac{1}{p}} \\
=\frac{t}{\varepsilon ^{\theta }}e^{\theta }\lim_{n\rightarrow \infty
}(n+1)^{\theta -\frac{1}{p}}=0,
\end{multline*}

provided $\theta -\frac{1}{p}<0,$which proves the result.
\end{proof}

\begin{corollary}
\label{SequenceLemma} For any $t,p>0$ and $\theta \in \lbrack 0,\frac{1}{p})$
we have 
\begin{equation*}
\lim_{n\rightarrow \infty }\frac{t^{n}}{\varepsilon ^{\theta n}}\frac{%
n^{\theta n}}{\left( n!\right) ^{1/p}}=0.
\end{equation*}
\end{corollary}

\begin{corollary}
\label{conv}We have 
\begin{equation*}
\lim_{n\rightarrow \infty }\sum_{k=n}^{m}\frac{\sqrt[p]{\hat{L}(T)^{k}T^{k}}%
}{(\beta -\alpha )^{k/q}}\frac{k^{\theta k}}{\sqrt[p]{k!}}=0
\end{equation*}%
for any $\alpha <\beta $ and $q>p$.
\end{corollary}

\section{Proof of the existence and uniqueness\label{proof}}

\label{EU} We now prove an important result which will immediately allow us
to establish Theorem \ref{KeyResult}.

\begin{theorem}
\ \label{MainEUTheorem} There exists a unique element $\xi _{0}\in 
\underline{\,Z}$ such that for all $t\in \mathcal{T}$ we have $\mathcal{T}%
(\xi _{0})(t)=\xi _{0}(t)$ almost everywhere. Moreover and for all $\xi \in
Z_{\alpha _{\ast }}$ 
\begin{equation*}
\lim_{n\rightarrow \infty }\mathcal{T}^{n}(\xi )=\xi _{0},
\end{equation*}%
is true in $Z_{\alpha }$ for all $\alpha \in (\alpha _{\ast },\overline{%
\alpha })$.
\end{theorem}

\begin{proof}
Let us fix $\xi \in Z_{\alpha _{\ast }}$. Lemma \ref{SecondCauchyLemma} and
Corollary \ref{conv} show that the sequence $\{\mathcal{T}^{n}(\xi
)\}_{n=1}^{\infty }$ is Cauchy in $Z_{\beta }$ and therefore converges in $%
Z_{\beta }$ for all $\beta >\alpha _{\ast }$. Thus there exists $\xi _{0}\in 
\underline{Z}=\cap _{\beta >\alpha _{\ast }}Z_{\beta }$ such that 
\begin{equation*}
\lim_{n\rightarrow \infty }\mathcal{T}^{n}(\xi )=\xi _{0},
\end{equation*}%
where the convergence takes place in $Z_{\beta }$ for all $\beta >\alpha
_{\ast }$.

We can now fix arbitrary $\delta <\beta $ and observe that $\mathcal{T}%
:Z_{\delta }\rightarrow Z_{\beta }$ is continuous. Therefore, passing to the
limit in both sides of the equality 
\begin{equation*}
\mathcal{T}(\mathcal{T}^{n}(\xi ))=\mathcal{T}^{n+1}(\xi )\in Z_{\beta }
\end{equation*}
we can conclude that 
\begin{equation*}
\mathcal{T}(\xi _{0})=\xi _{0}\text{ in }Z_{\beta }\text{ for any }\beta \in 
\mathcal{A},
\end{equation*}%
which implies that for all $t\in \mathcal{T}$ we have $\mathcal{T}(\xi
_{0})(t)=\xi _{0}(t)$ almost everywhere.

Finally, suppose there exists another element $\eta _{0}\in \underline{\,Z}$
such that for all $t\in \mathcal{T}$ we have $\mathcal{T}(\eta _{0})(t)=\eta
_{0}(t)$ almost everywhere. Then we see from the inequality (\ref{n-est})
that, 
\begin{equation*}
||\xi _{0}-\eta _{0}||_{Z_{\beta }}^{p}=||\mathcal{T}^{n}(\xi _{0})-\mathcal{%
T}^{n}(\eta _{0})||_{Z_{\beta }}^{p}\leq \frac{\left( a_{p}LT\right) ^{n}}{%
(\alpha -\phi _{1})^{\theta n}}\frac{n^{\theta n}}{\sqrt[2p]{n!}}||\xi
_{0}-\eta _{0}||_{Z_{\alpha }}\rightarrow 0,\ n\rightarrow \infty .
\end{equation*}%
Thus $||\xi _{0}-\eta _{0}||_{Z_{\beta }}\ =0$ for any $\beta \in \mathcal{A}
$. Hence $\xi _{0}$ is unique and the proof is complete.
\end{proof}

\textbf{The proof of the first two statements of Theorem \ref{KeyResult}}
follows immediately from Theorem \ref{MainEUTheorem} above by letting $\xi
\equiv u_{0}$.

\section{Continuity of the solution\label{contin}}

Let $\xi (t),t>0$, be the solution of equation (\ref{SDE-int}) constructed
in Theorems \ref{MainEUTheorem} and \ref{KeyResult}.

\begin{theorem}
For any $\alpha \in \mathcal{A}$, process $\xi $ has a continuous
modification $\eta (t)\in X_{\alpha }$, $t\in \lbrack 0,T]$, which solves
equation (\ref{SDE-int}).
\end{theorem}

\begin{proof}
Let us fix any $\beta \in \left( \alpha _{\ast },\alpha ^{\ast }\right] $
and $p\in \left[ 2,q\right) $. We can prove the existence of a continuous
modification by an application of Kolmogorov's continuity theorem in a
rather standard way. Indeed, using (\ref{SDE-int}) and the arguments similar
to those used in the proof of (\ref{EqAux1}) and (\ref{2ndIntMapLCondition})
we obtain%
\begin{eqnarray*}
\mathbb{E}||\xi (t)-\xi (s)||_{X_{\beta }}^{p} &\leq &\mathbb{E}\left[
||\int_{s}^{t}F(\xi (\tau ))d\tau +\int_{s}^{t}\Phi (\xi (\tau ))dW(\tau
)||_{X_{\alpha }}^{p}\right]  \\
&\leq &\frac{C(t-s)}{(\beta -\alpha )^{p\theta }}\left\Vert \xi \right\Vert
_{Z_{a,T}^{p}}^{p},\ 0\leq s<t\leq T,
\end{eqnarray*}%
where, for $\tau >0$,%
\begin{equation*}
C(\tau )=(\tau ^{p}+\left[ \frac{p}{2}(p-1)\right] ^{p/2}\tau
^{p/2})2^{p-1}L^{p}\leq (T^{p/2}+\left[ \frac{p}{2}(p-1)\right]
^{p/2})2^{p-1}L^{p}~\tau ^{p/2}.
\end{equation*}%
So we obtain the estimate%
\begin{equation*}
\mathbb{E}||\xi (t)-\xi (s)||_{X_{\beta }}^{p}=k(\xi ,T)\left\vert
t-s\right\vert ^{p/2}
\end{equation*}%
with $k(\xi ,T)=(T^{p/2}+\left[ \frac{p}{2}(p-1)\right] ^{p/2})2^{p-1}L^{p}%
\left\Vert \xi \right\Vert _{Z_{a}}^{p}$. The existence of a continuous
modification $\eta (t)\in X_{\beta }$ follows from Kolmogorov's continuity
theorem.

Since $\xi $ satisfies (\ref{SDE-int}) and $\eta (t)=\xi (t)$ a.s. we have%
\begin{equation*}
\eta (t)=\mathcal{T(\xi )(}t\mathcal{)}\text{ a.s.}
\end{equation*}%
Observe now that by (\ref{LipZ}) we have 
\begin{equation*}
\mathbb{E}\left[ ||\mathcal{T(\xi )(}t\mathcal{)-T(\eta )(}t\mathcal{)}%
||_{X_{\beta }}^{p}\right] \leq \frac{C(t)}{(\beta -\alpha )^{p/q}}%
\left\Vert \xi -\eta \right\Vert _{Z_{a,T}^{p}}^{p}=0,\ 0\leq t\leq T,
\end{equation*}%
which implies that $\mathcal{T(\xi )(}t\mathcal{)=T(\eta )(}t\mathcal{)}$
a.s. So we proved that%
\begin{equation*}
\eta (t)=\mathcal{T(}\eta \mathcal{)(}t\mathcal{)}\text{ a.s.}
\end{equation*}%
This equality holds in $X_{\beta }$ for any $\beta \in \left( \alpha _{\ast
},\alpha ^{\ast }\right] $. The proof is complete.
\end{proof}

\begin{remark}
Observe that $\left\Vert \xi -\eta \right\Vert _{Z_{\alpha ,T}^{p}}=0$ for
any $\alpha $, so the processes $\xi $ and $\eta $ coincide as elements of $%
\underline{Z_{T}^{p}}$.
\end{remark}

\section{Estimate of the solution\label{estofsol}}

\label{ES}

In this section, we derive a norm estimate of the solution $\xi $ from
Theorem \ref{MainEUTheorem}.

\begin{lemma}
\label{MainEULemma}For any $\alpha <\beta \in \mathcal{A}$ we have%
\begin{equation*}
||\mathcal{\xi }||_{Z_{\beta }}\leq E^{(p)}\left( \sqrt[p]{\hat{L}(T)T}%
,\beta -\alpha ,q^{-1}\right) \left( 1+||\zeta _{\alpha _{\ast
}}||_{Z_{\alpha _{\ast }}}\right) ^{p}.
\end{equation*}
\end{lemma}

\begin{proof}
Consider the approximating sequence $\left\{ \xi _{n}\right\} \subset 
\underline{Z}$ defined by 
\begin{equation*}
\xi _{n}=\mathcal{T}^{n}\mathcal{(}\zeta _{\alpha _{\ast }}),\ n=1,2,...
\end{equation*}%
We can use inequality (\ref{n-est0}) and further estimate its right-hand
side in the following way. We have 
\begin{equation*}
||\zeta _{\alpha _{\ast }}-\mathcal{T}(\zeta _{\alpha _{\ast
}})(s)||_{X_{\alpha }}^{p}=\mathbb{E}\left[ ||\int_{0}^{t}F(\zeta _{\alpha
_{\ast }})ds+\int_{0}^{t}\Phi (\zeta _{\alpha _{\ast }})dW(s)||_{X_{\alpha
}}^{p}\right] .
\end{equation*}%
Remark \ref{rem-bound} combined with the arguments similar to those used in
the proof of (\ref{EqAux1}) and (\ref{2ndIntMapLCondition}) implies the
estimate%
\begin{multline*}
\mathbb{E}\left[ ||\mathcal{T}^{n}(\zeta _{\alpha _{\ast }})(t)-\mathcal{T}%
^{n+1}(\zeta _{\alpha _{\ast }})(t)||_{X_{\beta }}^{p}\right]  \\
\leq \frac{n^{np/q}\hat{L}(T)}{(\alpha -\alpha _{\ast })^{p/q}}\left[ \frac{%
\hat{L}(T)}{(\beta -\alpha )^{p/q}}\right] ^{n}\frac{T^{n+1}}{\left(
n+1\right) !}\left( 1+||\zeta _{\alpha _{\ast }}||_{X_{\alpha _{\ast
}}}\right) ^{p}.
\end{multline*}%
In particular, we can set $\alpha =\psi _{n}=\beta -\frac{\beta -\alpha
_{\ast }}{n+1}=\alpha _{\ast }+n\frac{\beta -\alpha _{\ast }}{n+1}$. A
direct calculation shows that the above inequality transforms into%
\begin{multline*}
\mathbb{E}\left[ ||\mathcal{T}^{n}(\zeta _{\alpha _{\ast }})(t)-\mathcal{T}%
^{n+1}(\zeta _{\alpha _{\ast }})(t)||_{X_{\beta }}^{p}\right]  \\
\leq \left[ \frac{\hat{L}(T)(n+1)^{p/q}}{(\beta -\alpha _{\ast })^{p/q}}%
\right] ^{n+1}\frac{T^{n+1}}{\left( n+1\right) !}\left( 1+||\zeta _{\alpha
_{\ast }}||_{X_{\alpha _{\ast }}}\right) ^{p},
\end{multline*}%
which implies that%
\begin{equation*}
||\mathcal{T}^{n}(\zeta _{\alpha _{\ast }})-\mathcal{T}^{n+1}(\zeta _{\alpha
_{\ast }})||_{Z_{\beta }}^{p}\leq \frac{(n+1)^{(n+1)p/q}}{\left( n+1\right) !%
}\left[ \frac{\hat{L}(T)T}{(\beta -\alpha _{\ast })^{p/q}}\right]
^{n+1}\left( 1+||\zeta _{\alpha _{\ast }}||_{X_{\alpha _{\ast }}}\right)
^{p}.
\end{equation*}%
Then%
\begin{equation*}
||\zeta _{\alpha _{\ast }}-\mathcal{T}^{m}(\zeta _{\alpha _{\ast
}})||_{Z_{\beta }}\leq \sum_{n=1}^{m+1}\frac{n^{np/q}}{n!}\left[ \frac{\hat{L%
}(T)T}{(\beta -\alpha _{\ast })^{p/q}}\right] ^{n}\left( 1+||\zeta _{\alpha
_{\ast }}||_{X_{\alpha _{\ast }}}\right) ^{p}.
\end{equation*}%
Passing to the limit as $m\rightarrow \infty $ we obtain the bound%
\begin{equation*}
||\zeta _{\alpha _{\ast }}-\mathcal{\xi }||_{Z_{\beta }}\leq
\sum_{n=1}^{\infty }\frac{n^{np/q}}{n!}\left[ \frac{\hat{L}(T)T}{(\beta
-\alpha _{\ast })^{p/q}}\right] ^{n}\left( 1+||\zeta _{\alpha _{\ast
}}||_{X_{\alpha _{\ast }}}\right) ^{p}.
\end{equation*}%
Therefore%
\begin{multline*}
||\mathcal{\xi }||_{Z_{\beta }}\leq ||\zeta _{\alpha _{\ast }}||_{X_{\beta
}}+\sum_{n=1}^{\infty }\frac{n^{np/q}}{n!}\left[ \frac{\hat{L}(T)T}{(\beta
-\alpha _{\ast })^{p/q}}\right] ^{n}\left( 1+||\zeta _{\alpha _{\ast
}}||_{X_{\alpha _{\ast }}}\right) ^{p} \\
\leq \left( 1+\sum_{n=1}^{\infty }\frac{n^{np/q}}{n!}\left[ \frac{\hat{L}(T)T%
}{(\beta -\alpha _{\ast })^{p/q}}\right] ^{n}\right) \left( 1+||\zeta
_{\alpha _{\ast }}||_{X_{\alpha _{\ast }}}\right) ^{p} \\
=E^{(p)}\left( \sqrt[p]{\hat{L}(T)T},\beta -\alpha ,q^{-1}\right) \left(
1+||\zeta _{\alpha _{\ast }}||_{X_{\alpha _{\ast }}}\right) ^{p},
\end{multline*}%
which completes the proof.
\end{proof}

\section{Stochastic spin dynamics of a quenched particle system \label{spins}%
}

Our main example is motivated by the study of stochastic dynamics of
interacting particle systems. We follow the scheme of paper \cite{Dal},
adapted to our present setting, which allows to show the existence of
solutions with arbitrary large lifetime and their path-continuity. 

Let $\gamma \subset X={\mathbb{R}}^{d}$ be a locally finite set
(configuration) representing a collection of point particles. Each particle
with position $x\in X$ is characterized by an internal parameter (spin) $%
\sigma _{x}\in S={\mathbb{R}}^{1}$.

We fix a configuration $\gamma $ and look at the time evolution of spins $%
\sigma _{x}(t)$, $x\in \gamma $, which is described by a system of
stochastic differential equations in $S$ of the form 
\begin{equation}
d\sigma _{x}(t)=f_{x}(\bar{\sigma})dt+B_{x}(\bar{\sigma})dW_{x}(t),\ x\in
\gamma ,  \label{system}
\end{equation}%
where $\bar{\sigma}=(\sigma _{x})_{x\in \gamma }$ and $W=(W_{x})_{x\in
\gamma }$ is a collection of independent Wiener processes in $S$. We assume
that both drift and diffusion coefficients $f_{x}$ and $B_{x}$ depend only
on spins $\sigma _{y}$ with $\left\vert y-x\right\vert <r$ for some fixed
interaction radius $r>0$ and have the form 
\begin{equation}
f_{x}(\bar{\sigma})=\sum_{y\in \gamma }\varphi _{xy}(\sigma _{x},\sigma
_{y}),\ \ B_{x}(\bar{\sigma})=\sum_{y\in \gamma }\Psi _{xy}(\sigma
_{x},\sigma _{y}),  \label{system-coef}
\end{equation}%
where the mappings $\varphi _{xy}:S\times S\rightarrow S$ and $\Psi
_{xy}:S\times S\rightarrow S$ satisfy finite range and uniform Lipschitz
conditions, see Definition \ref{def-admiss} and Condition \ref{cond-inter1}
below.

Our aim is to realize system (\ref{system}) as an equation in a suitable
scale of Hilbert spaces and apply the results of previous sections in order
to find its strong solutions.

We introduce the following notations:

- $S^{\gamma }:=\prod _{x\in \gamma }S_{x}\ni \bar {\sigma }=(\sigma
_{x})_{x\in \gamma },\ \sigma _{x}\in S_{x}=S$;

- $\gamma _{x,r}:=\left \{y\in \gamma :\left \vert x-y\right \vert
<r\right
\},\ x\in \gamma $;

- $n_{x}\equiv n_{x,r}(\gamma ):=$ number of points in $\gamma _{x,r}$ ( $=$
number of particles interacting with particle in position $x$).

Observe that, although the number $n_{x}$ is finite, it is in general
unbounded function of $x$. We assume that it satisfies the following
regularity condition.

From now on, we assume that the following condition holds.

\begin{condition}
\label{cond-gamma}There exist constants $q>2$ and $a(\gamma ,r,q)>0$ such
that 
\begin{equation}
n_{x,r}(\gamma )\leq a(\gamma ,r,q)\left( 1+\left\vert x\right\vert \right)
^{1/q}  \label{c-gamma}
\end{equation}%
for all $x\in X$.
\end{condition}

\begin{remark}
Condition (\ref{c-gamma}) holds if $\gamma $ is a typical realization of a
Poisson or Gibbs (Ruelle) point process in $X$. For such configurations,
stronger (logarithmic) bound holds: 
\begin{equation*}
n_{x,r}(\gamma )\leq c(\gamma )\left[ 1+\log (1+\left\vert x\right\vert )%
\right] r^{d},
\end{equation*}%
see e.g. \cite{R70} and \cite[p.~1047]{K93}. Thus (\ref{c-gamma}) holds for
any $q>0$.
\end{remark}

\subsection{Existence of the dynamics}

Our dynamics will live in the scale of Hilbert spaces 
\begin{equation*}
X_{\alpha }=S_{\alpha }^{\gamma }:=\left\{ \bar{q}\in S^{\gamma }:\left\Vert 
\bar{q}\right\Vert _{\alpha }:=\sqrt{\sum_{x\in \gamma }\left\vert
q_{x}\right\vert ^{2}e^{-\alpha \left\vert x\right\vert }}<\infty \right\}
,\ 0<\alpha _{\ast }<\alpha <\alpha ^{\ast }.
\end{equation*}%
We fix the parameters $\alpha _{\ast }$ and $\alpha ^{\ast }$, which can be
chosen in an arbitrary way.

We set 
\begin{equation*}
{\mathcal{H}}=S_{0}^{\gamma }:=\left\{ \bar{q}\in S^{\gamma }:\left\Vert 
\bar{q}\right\Vert _{0}:=\sqrt{\sum_{x\in \gamma }\left\vert
q_{x}\right\vert ^{2}}<\infty \right\}
\end{equation*}%
and define the corresponding spaces ${\mathcal{GL}}_{p}(\mathfrak{X})$ and ${%
\mathcal{GL}}_{p}(\mathfrak{X},\mathfrak{H})$ (cf. Definition \ref{Defovsop}%
). Observe that $W(t):=\left( W_{x}(t)\right) _{x\in \gamma }$ is a cylinder
Wiener process in $\mathcal{H}$.

\bigskip

Let ${\mathcal{V}}$ be a family of mappings $V_{xy}:S^{2}\rightarrow S$, $%
x,y\in \gamma $.

\begin{definition}
\label{def-admiss}We call the family $\mathcal{V}$ admissible if it
satisfies the following two assumptions:
\end{definition}

\begin{itemize}
\item finite range: there exists constant $r>0$ such that $V_{xy}\equiv 0$
if $\left\vert x-y\right\vert \geq r$;

\item uniform Lipschitz continuity: there exists constant $C>0$ such that 
\begin{equation}
\left\vert V_{xy}(q_{1}^{\prime },q_{2}^{\prime })-V_{xy}(q_{1}^{\prime
\prime },q_{2}^{\prime \prime })\right\vert \leq C\left( \left\vert
q_{1}^{\prime }-q_{1}^{\prime \prime }\right\vert +\left\vert q_{2}^{\prime
}-q_{2}^{\prime \prime }\right\vert \right)  \label{LipS2}
\end{equation}%
for all $x,y\in \gamma $ and $q_{1}^{\prime},q_{2}^{\prime},q_{1}^{\prime
\prime },q_{2}^{\prime \prime }\in S$.
\end{itemize}

Define a map $\overline{V}:S^{\gamma }\rightarrow S^{\gamma }$ and a linear
operator $\widehat{V}(\bar{q}):S^{\gamma }\rightarrow S^{\gamma }$, $\bar{q}%
\in S^{\gamma }$, by the formula 
\begin{equation*}
\overline{V}_{x}(\bar{q})=\sum_{y\in \gamma }V_{xy}(q_{x},q_{y}),
\end{equation*}%
and 
\begin{equation*}
\left( \widehat{V}(\bar{q}){\bar{\sigma}}\right) _{x}:=\overline{V}_{x}(\bar{%
q})\sigma _{x},\,x\in \gamma ,\,{\bar{\sigma}}\in S^{\gamma },
\end{equation*}%
respectively.

\begin{lemma}
\label{lemma-gl}Assume that ${\mathcal{V}}$ is admissible. Then $\overline{V}%
\in {\mathcal{GL}}_{q}(\mathfrak{X})$ and $\widehat{V}\in {\mathcal{GL}}_{q}(%
\mathfrak{X},\mathfrak{H})$.
\end{lemma}

The proof of this Lemma is quite tedious and will be given in Section \ref%
{sec-lemma-gl} below.

\bigskip

Now we can return to the discussion of system (\ref{system}). Assume that
the following condition holds.

\begin{condition}
\label{cond-inter1} The families of mappings $\left\{ \varphi _{xy}\right\}
_{x,y\in \gamma }$ and $\left\{ \Psi _{xy}\right\} _{x,y\in \gamma }$ from (%
\ref{system-coef}) are admissible.
\end{condition}

By Lemma \ref{lemma-gl} we have $\overline{\varphi }\in {\mathcal{GL}}_{p}(%
\mathfrak{X})$ and $\widehat{\Psi }\in {\mathcal{GL}}_{p}(\mathfrak{X},%
\mathfrak{H})$. Thus we can write (\ref{system}) in the form 
\begin{equation*}
\bar{\sigma}(t)=\overline{\varphi }(\bar{\sigma})dt+\widehat{\Psi }(\bar{%
\sigma})dW(t),
\end{equation*}%
where $W(t)=\left( W_{x}(t)\right) _{x\in \gamma }$, and apply the results
of the previous sections to its integral counterpart. We summarize those
results in the following theorem, which follows directly from Theorem \ref%
{theor-main}.

\begin{theorem}
\label{theor-particles}Assume that Conditions \ref{cond-gamma} and \ref%
{cond-inter1} hold. Then, for any $\alpha >0$, $\bar{\sigma}_{0}\in
X_{\alpha }$, $p\in \left[ 2,q\right) $ and $T>0$, system (\ref{system}) has
a unique strong solution $u\in Z_{\beta ,T}^{p}$, for any $\beta >\alpha $.
This solution has continuous sample paths a.s.
\end{theorem}

This result implies of course that, for each $x\in \gamma $, equation (\ref%
{system}) has a path-continuous strong solution, which is unique in the
class of progressively measurable square-integrable processes.

\subsection{Proof of Lemma \protect\ref{lemma-gl}. \label{sec-lemma-gl}}

This proof is a modification of the proof given in \cite{Dal} for $q=2$.

\textbf{Step 1. }We first show that $\overline{V}$ is a mapping $S_{\alpha
}^{\gamma }\rightarrow S_{\beta }^{\gamma }$ for any $\alpha <\beta $. For
any $\bar{q}\in S_{\alpha }^{\gamma }$ we have%
\begin{eqnarray*}
\left\Vert \overline{V}(\bar{q})\right\Vert _{\beta }^{2} &=&\sum_{x\in
\gamma }\left\vert \sum_{y\in \gamma }V_{xy}(q_{x},q_{y})\right\vert
^{2}e^{-\beta \left\vert x\right\vert } \\
&\leq &3C^{2}\sum_{x\in \gamma }\sum_{y\in \gamma _{x,r}}n_{x}\left(
1+\left\vert q_{x}\right\vert ^{2}+\left\vert q_{y}\right\vert ^{2}\right)
e^{-\beta \left\vert x\right\vert }.
\end{eqnarray*}%
The polynomial bound on the growth of $n_{x}$ implies that 
\begin{equation*}
\sum_{x\in \gamma }\sum_{y\in \gamma _{x,r}}n_{x}e^{-\beta \left\vert
x\right\vert }=\sum_{x\in \gamma }n_{x}^{2}e^{-\beta \left\vert x\right\vert
}\leq \sum_{x\in \gamma }n_{x}^{2}e^{-\alpha _{\ast }\left\vert x\right\vert
}=:c(\gamma ,\alpha _{\ast })<\infty .
\end{equation*}%
Next, we estimate%
\begin{multline*}
\sum_{x\in \gamma }\sum_{y\in \gamma _{x,r}}n_{x}\left\vert q_{x}\right\vert
^{2}e^{-\beta \left\vert x\right\vert }=\sum_{x\in \gamma
}n_{x}^{2}\left\vert q_{x}\right\vert ^{2}e^{-\left( \beta -\alpha \right)
\left\vert x\right\vert }e^{-\alpha \left\vert x\right\vert } \\
\leq \sup_{x\in \gamma }\left( n_{x}^{2}e^{-\left( \beta -\alpha \right)
\left\vert x\right\vert }\right) \left\Vert \bar{q}\right\Vert _{\alpha
}^{2}.
\end{multline*}%
Observe that $\sum\limits_{x\in \gamma }\sum\limits_{y\in \gamma
_{x,r}}=\sum\limits_{\substack{ x,y\in \gamma  \\ \left\vert x-y\right\vert
<r}}=\sum\limits_{y\in \gamma }\sum\limits_{x\in \gamma _{y,r}}$, and so 
\begin{multline*}
\sum_{x\in \gamma }\sum_{y\in \gamma _{x,r}}n_{x}\left\vert q_{y}\right\vert
^{2}e^{-\beta \left\vert x\right\vert }\leq e^{\beta r}\sum_{y\in \gamma
}N_{y}\left\vert q_{y}\right\vert ^{2}e^{-(\beta -\alpha )\left\vert
y\right\vert }e^{-\alpha \left\vert y\right\vert } \\
\leq e^{\beta r}\sup_{y\in \gamma }\left( N_{y}e^{-(\beta -\alpha
)\left\vert y\right\vert }\right) \left\Vert \bar{q}\right\Vert _{\alpha
}^{2},
\end{multline*}%
where $N_{y}:=\sum_{x\in \gamma _{y,r}}n_{x}$. Here we used inequality $%
\left\vert y\right\vert \leq \left\vert y-x\right\vert +\left\vert
x\right\vert \leq r+\left\vert x\right\vert $ for $y\in \gamma _{x,r}$, so
that $e^{-\beta \left\vert x\right\vert }\leq e^{\beta r}e^{-\beta
\left\vert y\right\vert }$. Condition \ref{cond-gamma} implies that%
\begin{equation*}
N_{x}\leq a(\gamma ,r,q)^{2}\left( 1+\left\vert x\right\vert \right)
^{1/q}\left( 1+r+\left\vert x\right\vert \right) ^{1/q}<c(\gamma ,r,q)\left(
1+\left\vert x\right\vert \right) ^{2/q},
\end{equation*}%
for some constant $c(\gamma ,r,q)>0$, and 
\begin{equation*}
n_{x}^{2}\leq a(\gamma ,r,q)^{2}\left( 1+\left\vert x\right\vert \right)
^{2/q}
\end{equation*}%
for any $x\in \gamma $. Eventually we obtain the bound 
\begin{equation*}
\left\Vert \overline{V}(\bar{q})\right\Vert _{\beta }^{2}\leq L^{2}\left[
\sup_{s>0}(1+s)e^{-(\beta -\alpha )s}\right] ^{2/q}\left\Vert \bar{q}%
\right\Vert _{\alpha }^{2}\leq L^{2}\left( \beta -\alpha \right)
^{-2/q}\left\Vert \bar{q}\right\Vert _{\alpha }^{2}<\infty ,\ L<\infty .
\end{equation*}

\textbf{Step 2. }Lipschitz condition (\ref{LipS2}) implies the estimate 
\begin{multline*}
\left\Vert \overline{V}(\bar{q}^{\prime })-\overline{V}(\bar{q}^{\prime
\prime })\right\Vert _{\beta }^{2}=\sum_{x\in \gamma }\left\vert \sum_{y\in
\gamma }V_{xy}(q_{x}^{\prime },q_{y}^{\prime })-\sum_{y\in \gamma
}V_{xy}(q_{x}^{\prime \prime },q_{y}^{\prime \prime })\right\vert
^{2}e^{-\beta \left\vert x\right\vert } \\
\leq 2C^{2}\sum_{x\in \gamma }\sum_{y\in \gamma _{x,r}}n_{x}\left(
\left\vert q_{x}^{\prime }-q_{x}^{\prime \prime }\right\vert ^{2}+\left\vert
q_{y}^{\prime }-q_{y}^{\prime \prime }\right\vert ^{2}\right) e^{-\beta
\left\vert x\right\vert }
\end{multline*}%
for any\textbf{\ }$\bar{q}^{\prime },\bar{q}^{\prime \prime }\in S_{\alpha
}^{\gamma }$. Similar to Step 1, we obtain the bound 
\begin{multline*}
\left\Vert \overline{V}(\bar{q}^{\prime })-\overline{V}(\bar{q}^{\prime
\prime })\right\Vert _{\beta }^{2}\leq L^{2}\left[ \sup_{s>0}(1+s)e^{-(\beta
-\alpha )s}\right] ^{2/q}\left\Vert \bar{q}^{\prime }-\bar{q}^{\prime \prime
}\right\Vert _{\alpha }^{2} \\
\leq L^{2}\left( \beta -\alpha \right) ^{-2/q}\left\Vert \bar{q}^{\prime }-%
\bar{q}^{\prime \prime }\right\Vert _{\alpha }^{2}<\infty ,\ L<\infty .
\end{multline*}

\textbf{Step 3.} The inclusion $\overline{V}(\bar{q})\in S_{\beta }^{\gamma
} $ implies that $\widehat{V}(\bar{q}){\bar{\sigma}}\in S_{\beta }^{\gamma }$
for any ${\bar{\sigma}}\in \mathcal{H}=S_{0}^{\gamma }$. A direct
calculation shows that $\widehat{V}(\bar{q}):\mathcal{H}\rightarrow {S}%
_{\beta }^{\gamma }$ is a Hilbert-Schmidt operator with the norm equal to $%
\left\Vert \bar{V}(\bar{q})\right\Vert _{\beta }$. Thus the inclusion $%
\overline{V}\in {\mathcal{GL}}^{(1)}$ implies that $\widehat{V}\in {\mathcal{%
GL}}^{(2)}$. \hfill $\square $

\section{Further examples\label{examples}}

In this section we give two examples of linear maps of the class ${\mathcal{%
GL}}_{q}(\mathfrak{B})$.

\textbf{Example 1.} Let $B_{\alpha }:=L^{p}(\mathbb{R}^{1},e^{-\alpha
\left\vert x\right\vert }dx)$, $p>1$, and $f(u)=Au$, where $A$ is the
integral operator with kernel $K$:%
\begin{equation*}
Au(x)=\int K(x,y)u(y)dy,\ x,y\in \mathbb{R}^{1}.
\end{equation*}

\begin{condition}
\label{c1}There exist $\beta ^{\ast }>\alpha ^{\ast }$ and $a>0$ such that%
\begin{equation*}
\left\vert K(x,y)\right\vert \leq ae^{-\frac{\beta ^{\ast }}{p}\left\vert
x-y\right\vert }\left( 1+\left\vert y\right\vert \right) ^{\delta },\ \delta
>0,
\end{equation*}%
for a.a. $x\in \mathbb{R}$.
\end{condition}

\begin{remark}
It is clear that $K(x,y)$ can grow to infinity along the main diagonal $x=y$%
, which implies that $A$ is in general unbounded in any weighted $L^{p}$.
\end{remark}

\begin{proposition}
\label{p1}Assume that Condition \ref{c1} holds. Then $A\in {\mathcal{GL}}%
_{q}(\mathfrak{B})$ with $q=\frac{p-1}{p\delta }$.
\end{proposition}

\begin{remark}
For an implementation of any version of Ovsyannikov-type method, we need $%
q\geq 1$, which implies $\delta \leq \frac{p-1}{p}<1.$
\end{remark}

\noindent \textbf{Proof. }We start with the following estimate of the norm
of operator $A$: 
\begin{multline*}
\left\Vert Au\right\Vert _{B_{\beta }}^{p}=\int \left[ \int K(x,y)u(y)dy%
\right] ^{p}e^{-\beta \left\vert x\right\vert }dx \\
\leq a^{p}\int \left[ \int e^{-\frac{\beta ^{\ast }}{p}\left\vert
x-y\right\vert }\left( 1+\left\vert y\right\vert \right) ^{\delta
}\left\vert u(y)\right\vert dy\right] ^{p}~e^{-\beta \left\vert x\right\vert
}dx \\
=a^{p}\int \left[ \int e^{-\varepsilon \left\vert x-y\right\vert }\left(
1+\left\vert y\right\vert \right) ^{\delta }\left\vert u(y)\right\vert e^{-%
\frac{\beta }{p}\left\vert x-y\right\vert }dy\right] ^{p}~e^{-\beta
\left\vert x\right\vert }dx,
\end{multline*}%
where $\varepsilon =\frac{\beta ^{\ast }-\beta }{p}$. Observe that 
\begin{equation*}
e^{-\frac{\beta }{p}\left\vert x-y\right\vert }e^{-\frac{\beta }{p}%
\left\vert x\right\vert }\leq e^{-\frac{\beta }{p}\left\vert y\right\vert },
\end{equation*}%
so that%
\begin{equation*}
\left\Vert Au\right\Vert _{B_{\beta }}^{p}\leq a^{p}\int \left[ \int
e^{-\varepsilon \left\vert x-y\right\vert }\left( 1+\left\vert y\right\vert
\right) ^{\delta }\left\vert u(y)\right\vert e^{-\frac{\beta }{p}\left\vert
y\right\vert }dy\right] ^{p}dx.
\end{equation*}%
For $\theta $ such that $\theta ^{-1}+p^{-1}=1$ we have%
\begin{equation*}
e^{-\varepsilon \left\vert x-y\right\vert }\left( 1+\left\vert y\right\vert
\right) ^{\delta }\left\vert u(y)\right\vert e^{-\frac{\beta }{p}\left\vert
y\right\vert }=\left[ e^{-\frac{\varepsilon }{\theta }\left\vert
x-y\right\vert }\left( 1+\left\vert y\right\vert \right) ^{\delta }e^{-\frac{%
\beta -\alpha }{p}\left\vert y\right\vert }\right] \times \left[ e^{-\frac{%
\varepsilon }{p}\left\vert x-y\right\vert }\left\vert u(y)\right\vert e^{-%
\frac{\alpha }{p}\left\vert y\right\vert }\right] .
\end{equation*}%
Then, by Holder's inequality,%
\begin{multline*}
\left\Vert Au\right\Vert _{B_{\beta }}^{p} \\
\leq a^{p}\int \left[ \left\{ \int \left[ e^{-\frac{\varepsilon }{\theta }%
\left\vert x-y\right\vert }\left( 1+\left\vert y\right\vert \right) ^{\delta
}e^{-\frac{\beta -\alpha }{p}\left\vert y\right\vert }\right] ^{\theta
}dy\right\} ^{p/\theta }\times \int \left[ e^{-\frac{\varepsilon }{p}%
\left\vert x-y\right\vert }\left\vert u(y)\right\vert e^{-\frac{\alpha }{p}%
\left\vert y\right\vert }\right] ^{p}dy\right] dx \\
=a^{p}\int \left[ \left\{ \int e^{-\varepsilon \left\vert x-y\right\vert
}\left( 1+\left\vert y\right\vert \right) ^{\theta \delta }e^{-\frac{\theta 
}{p}(\beta -\alpha )\left\vert y\right\vert }dy\right\} ^{p/\theta }\times
\int e^{-\varepsilon \left\vert x-y\right\vert }\left\vert u(y)\right\vert
^{p}e^{-\alpha \left\vert y\right\vert }dy\right] dx \\
\leq a^{p}b\left[ \int e^{-\varepsilon \left\vert x-y\right\vert }dy\right]
^{p/\theta }\times \int \int e^{-\varepsilon \left\vert x-y\right\vert
}\left\vert u(y)\right\vert ^{p}e^{-\alpha \left\vert y\right\vert }dydx \\
\leq a^{p}bc^{p/\theta }\int \int e^{-\varepsilon \left\vert x-y\right\vert
}\left\vert u(y)\right\vert ^{p}e^{-\alpha \left\vert y\right\vert }dydx,
\end{multline*}%
where%
\begin{equation*}
b=\sup_{s\geq 0}\left( 1+s\right) ^{\theta \delta }e^{-\frac{\theta }{p}%
(\beta -\alpha )s}
\end{equation*}%
and%
\begin{equation*}
c=\int e^{-\frac{\beta ^{\ast }-\beta }{p}\left\vert x-y\right\vert }dy>\int
e^{-\varepsilon \left\vert y\right\vert }dy.
\end{equation*}%
Observe that%
\begin{equation*}
\int \int e^{-\varepsilon \left\vert x-y\right\vert }\left\vert
u(y)\right\vert ^{p}e^{-\alpha \left\vert y\right\vert }dydx=\int
e^{-\varepsilon \left\vert x-y\right\vert }dy~\left\Vert u\right\Vert
_{B_{\alpha }}^{p}=c\left\Vert u\right\Vert _{B_{\alpha }}^{p},
\end{equation*}%
which leads to the bound%
\begin{equation*}
\left\Vert Au\right\Vert _{B_{\beta }}^{p}\leq a^{p}bc^{p/\theta
+1}\left\Vert u\right\Vert _{B_{\alpha }}^{p}.
\end{equation*}%
It remains to compute constant $b=\left[ \sup_{s\geq 0}\left( 1+s\right) e^{-%
\frac{1}{p\delta }(\beta -\alpha )s}\right] ^{\theta \delta }$. Equating to $%
0$ the derivative $\frac{\partial }{\partial s}\left( 1+s\right) e^{-\frac{1%
}{p\delta }(\beta -\alpha )s}$ we obtain%
\begin{equation*}
b=\frac{C}{(\beta -\alpha )^{\theta \delta }},
\end{equation*}%
for some constant $C>0.$

It is clear that estimate (\ref{lg}) holds with $q=\frac{1}{\theta \delta }=%
\frac{p-1}{p\delta }$. 
\hfill%
$\square $

\textbf{Example 2.} A somewhat similar example is given by the spaces of
sequences 
\begin{equation*}
B_{\alpha }:=\left\{ (u_{k})_{k\in \mathbb{Z}}:\sum_{k\in \mathbb{Z}%
}\left\vert u_{k}\right\vert ^{p}e^{-\alpha \left\vert k\right\vert }<\infty
\right\} ,\ p>1,
\end{equation*}%
and the linear map given an infinite matrix $A=(A_{kj})_{k,j\in \mathbb{Z}}$
is with elements satisfying the bound 
\begin{equation*}
\left\vert A_{kj}\right\vert \leq ae^{-\frac{\beta ^{\ast }}{p}\left\vert
k-j\right\vert }\left( 1+\left\vert j\right\vert \right) ^{\delta }
\end{equation*}%
for some $\beta ^{\ast }>\alpha ^{\ast }$, $a>0$ and all $k\in \mathbb{Z}$.
The proof of the inclusion $A\in {\mathcal{GL}}_{q}(\mathfrak{B})$, $q=\frac{%
p-1}{p\delta }$, is similar to that of Proposition \ref{p1}. Similar to the
previous example, we have in general%
\begin{equation*}
\left\vert A_{kk}\right\vert \rightarrow \infty ,\ k\rightarrow \infty ,
\end{equation*}%
so that operator $A$ is unbounded in any weighted $l^{p}$.

\end{document}